\documentclass{amsart}

\usepackage{amssymb}
\usepackage{color}
\usepackage{amsxtra} 

\newtheorem{theorem}{Theorem}
\newtheorem{lem}[theorem]{Lemma}
\newtheorem{prop}[theorem]{Proposition}
 
\newtheorem*{theorem*}{Main Theorem} 
\newtheorem*{theorem**}{Theorem} 

\theoremstyle{definition}

\theoremstyle{remark}
\newtheorem{remark}[theorem]{Remark}

\numberwithin{equation}{section}

\begin{document}

\title[The solution map 
for the Euler equations
]
{Loss of continuity of the solution map for the Euler equations in $\alpha$-modulation and H\"older spaces 
}

\author{Gerard Misio{\l}ek}
\address{Department of Mathematics, University of Colorado, Boulder, CO 80309-0395, USA 
and 
Department of Mathematics, University of Notre Dame, IN 46556, USA} 
\email{gmisiole@nd.edu} 

\author{Tsuyoshi Yoneda}
\address{Department of Mathematics, Tokyo Institute of Technology, Meguro-ku, Tokyo 152-8551, Japan} 
\email{yoneda@math.titech.ac.jp}


\date{\today} 



\begin{abstract} 
We study the incompressible Euler equations in the $\alpha$-modulation $M^{s,\alpha}_{p,q}$ 
and H\"older $C^{1+\sigma}$ spaces on the plane. 
We show that for these spaces the associated data-to-solution map is not continuous on bounded sets. 
\end{abstract} 

\maketitle

\section{Introduction} 
\label{sec:Intro} 

In this paper we study the Cauchy problem for the non-periodic Euler equations of 
incompressible hydrodynamics 
\begin{align*} 
&u_t + u \cdot \nabla u + \nabla p = 0, 
\qquad\quad\quad
t \geq 0, \; x \in \mathbb{R}^n 
\\  \tag{E} \label{E} 
&\mathrm{div}\, u = 0 
\\ 
&u(0) = u_0 
\end{align*} 
with initial data in the $\alpha$-modulation spaces. 
In particular, our results apply to the Besov spaces including the classical H\"older-Zygmund spaces. 
According to the standard notion of well-posedness due to Hadamard a Cauchy problem 
is said to be locally (in time) well-posed in a Banach space $X$ if given any initial data $u_0$ in $X$ 
there is a time $T>0$ and a unique solution $u$ in a Banach space $Y \subset C([0,T), X)$ 
which depends continuously on the initial data. 
Otherwise the Cauchy problem is said to be locally ill-posed. 

Ill-posedness results establishing loss of continuity of the solution map  
$u_0 \to u$ for the Euler equations in the $C^1$ space and the borderline Besov space $B^1_{\infty,1}$ 
have been proved recently in \cite{MY2}. 
Here we refine the techniques of that paper to obtain ill-posedness results of this type 
in $\alpha$-modulation spaces $M^{s,\alpha}_{p,q}$ for $1<s<2$ 
and in $C^{1+\sigma}$ for $0<\sigma<1$. 
More precisely, 
following the approach of Bourgain and Li \cite{BL} we construct a Lagrangian flow with a large gradient 
and then choose a suitable high-frequency perturbation of the initial vorticity to show 
that the assumption of continuity of the solution map $u_0 \to u$ in the above spaces 
necessarily leads to a contradiction with the results of Kato and Ponce \cite{KP0, KP}. 
We will work with the vorticity equations which in two dimensions assume the form 
\begin{align} \nonumber 
&\omega_t + u{\cdot}\nabla \omega = 0, 
\qquad\qquad\qquad\quad 
t \geq 0, \; x \in \mathbb{R}^2 
\\ \label{B2:eq:euler-v} 
&u = K \ast \omega = \nabla^\perp \Delta^{-1} \omega 
\\  \nonumber
&\omega(0) = \omega_0 
\end{align} 
where 
$$
K(x) = \frac{1}{2\pi} \frac{(-x_2, x_1)}{ |x|^2} 
\quad 
\text{and} 
\quad 
\nabla^\bot = \Big(-\frac{\partial}{\partial x_2}, \frac{\partial}{\partial x_1} \Big)
$$ 
denote the Biot-Savart kernel and the symplectic gradient, respectively. 

The first rigorous results on the Cauchy problem for the incompressible Euler equations go back to 
Gyunter \cite{Gu}, Lichtenstein \cite{Li} and Wolibner \cite{Wo}. 
A survey of those and numerous further results can be found for example in Majda and Bertozzi \cite{MB}, 
Constantin \cite{Co} or Bahouri, Chemin and Danchin \cite{BCD}. 
For the most recent progress on local ill-posedness in borderline spaces such as 
$C^1$, $W^{n/p+1,p}$, $B^{n/p+1}_{p,q}$ as well as $C^k$, $C^{k-1,1}$ for integer $k \geq1$ 
we refer to the papers of 
Bourgain and Li \cite{BL, BL1}, Elgindi and Masmoudi \cite{EM} and the authors \cite{MY2}. 
For earlier results in $C^\sigma$ with $0 < \sigma < 1$, 
$B^s_{p,\infty}$ for $s > 0$, $p > 2$ and $s > n(2/p-1)$, $1 \leq p \leq 2$ 
or the logarithmic Lipschitz spaces $\log\text{Lip}^\alpha$ for $0 < \alpha < 1$ 
we refer to Bardos and Titi \cite{BT}, Cheskidov and Shvydkoy \cite{CS} and the authors \cite{MY}. 

Our main goal is to prove the following result. 
%
%
\begin{theorem*} \label{M:thm} 
Let $0 < \sigma < 1$, $0 < \alpha \leq 1$, $2 \leq p \leq \infty$, $1 \leq q \leq \infty$ and 
suppose that $M^{1+\sigma,\alpha}_{p,q}$ is continuously embedded in $C^1$. 
Then the solution map of the incompressible Euler equations \eqref{E} is not continuous 
on bounded subsets of $M^{1+\sigma,\alpha}_{p,q}$. 
\end{theorem*} 
Thus, the Euler equations are in general locally ill-posed in $\alpha$-modulation spaces 
in the sense of  Hadamard given above.\footnote{For example, we have 
$M^{1+\sigma, \alpha}_{\infty,q} \subset C^1$ whenever $\sigma > 2(1-\alpha)(1 - 1/q)$.} 
For the definition of the $\alpha$-modulation spaces see Section \ref{B2:sec:Lag} below. 
\begin{remark} 
Observe that $M^{s,1}_{p,q}$ coincides with the usual Besov space $B^s_{p,q}$. Therefore, 
somewhat surprisingly, Theorem \ref{M:thm} also yields ill-posedness (in the sense that 
the data-to-solution map loses its continuity properties) 
even in the classical H\"older spaces $B^{1+\sigma}_{\infty,\infty} = C^{1+\sigma}$ for $0 < \sigma < 1$. 

Continuous dependence results for the Euler equations (in the strong topology) have been obtained 
for initial data in Sobolev spaces $H^s$ and more generally $W^{s,p}$ with $s > n/p + 1$ 
for example in Ebin and Marsden \cite{EM}, Kato and Lai \cite{KL} and Kato and Ponce \cite{KP}. 
However, this is a rather difficult part of the local well-posedness theory which has not yet been 
satisfactorily resolved. 
\end{remark} 
\begin{remark} 
A different mechanism involving a gradual loss of regularity of the solution map is described 
by Morgulis, Shnirelman and Yudovich  \cite{MSY}. 
\end{remark} 
\begin{remark} 
In this context it is also worth pointing out that neither for the critical Besov space $B^1_{\infty,1}(\mathbb{R}^n)$ 
nor for the space $B^{1+p/n}_{p,1}(\mathbb{R}^n)$ 
are the Euler equations strongly ill-posed in the sense of Bourgain and Li \cite{BL}. 
This can be seen by examining the arguments given in Pak and Park \cite{PP} and Vishik \cite{V}. 
\end{remark} 

Our general strategy will be similar to that employed in \cite{MY2} which we will use as the main reference. 
The remainder of the paper is organized as follows. 
In Section \ref{B2:sec:Lag} we describe the general set up and prove several technical lemmas. 
The whole of Section \ref{sec:proof} is then devoted to the proof of Theorem \ref{M:thm}. 
Although the constructions in Sections \ref{B2:sec:Lag} and \ref{sec:proof} will be carried out in 2D 
they can be readily adapted to the 3D case. 
Rather than doing that 
in Section \ref{shear flow} we give a direct proof of ill-posedness in $C^{1+s}$ by using 
a 3D shear flow argument.

\section{Basic Setup: Vorticity and Lagrangian Flow} 
\label{B2:sec:Lag}

We first recall the definition of $\alpha$-modulation spaces. For a more detailed account the reader 
is referred for example to \cite{FHW}. 
A countable set $\mathcal Q$ of subsets $Q\subset \mathbb{R}^n$ is called an admissible 
covering if  $\mathbb{R}^n=\bigcup_{Q\in\mathcal{Q}}Q$ and if there is $n_0 < \infty$
such that $\#\{Q'\in\mathcal{Q}: Q\cap Q'\not=\emptyset\}\leq n_0$ for all $Q\in\mathcal{Q}$.
Let 
\begin{eqnarray*}
r_Q&=&\sup\{r\in\mathbb R:B(c_r,r)\subset Q, c_r \in \mathbb{R}^n \}\\
R_Q&=&\inf\{R\in\mathbb{R}:Q\subset B(c_R,R), c_R \in \mathbb{R}^n \}.
\end{eqnarray*} 
Given $0\leq \alpha\leq 1$, an admissible covering is an $\alpha$-covering of $\mathbb{R}^n$ if 
$|Q|\sim (1+|x|^2)^{\alpha n/2}$ (uniformly) for all $Q\in\mathcal Q$ and all $x\in Q$ 
and where 
$\sup_{Q\in\mathcal Q}R_Q/r_Q\leq K$ 
for some $K<\infty$. 
Let $\mathcal Q$ be an $\alpha$-covering of $\mathbb{R}^n$. 
A bounded admissible partition of unity of order $p$ (abbreviated $p$-BAPU) 
corresponding to $\mathcal Q$ is a family of smooth functions $\{\psi_Q\}_{Q\in\mathcal Q}$ 
satisfying
\begin{eqnarray*}
& &
\psi_Q:\mathbb{R}^n\to[0,1],\quad \text{supp}\ \psi_Q\subset Q,\\
& &
\sum_{Q\in\mathcal Q}\psi_Q(\xi)\equiv 1,\quad \xi\in\mathbb{R}^n,\\
& &
\sup_{Q\in\mathcal Q}|Q|^{1/p-1}\|\mathcal F^{-1}\psi_Q\|_{L^p}<\infty 
\end{eqnarray*}
where $\mathcal{F}$ denotes the Fourier transform. 

For any $1\leq p,q\leq\infty$, $s\in \mathbb{R}$ and $0\leq \alpha\leq 1$ the $\alpha$-modulation space 
$M^{s,\alpha}_{p,q}(\mathbb{R}^n)$ is the space of all tempered distributions $f$ for which 
the following norm 
\begin{equation*} 
\|f\|_{M^{s,\alpha}_{p,q}} 
= \left\{ 
\begin{matrix} \displaystyle
~~~\Bigg( \sum_{Q\in\mathcal Q} \big( 1+| \xi_Q |^2 \big)^{qs/2} 
\big\|  \mathcal F^{-1}\psi_Q\mathcal F f  \big\|_{L^p(\mathbb{R}^n)}^q  \Bigg)^{1/q}& 
\mathrm{if} \quad 1\leq q <\infty 
\\  \displaystyle
\sup_{Q \in \mathcal{Q}}{ \big( 1 + | \xi_Q |^2 \big)^{s/2}} \big\| \mathcal F^{-1}\psi_Q\mathcal F f \big\|_{L^p}& 
\mathrm{if} \qquad\;\; q=\infty 
\end{matrix} 
\right. 
\end{equation*} 
is finite, 
where $\{ \xi_Q \in Q: Q \in \mathcal{Q} \}$ is an arbitrary sequence. 
One shows that this definition is independent of an $\alpha$-covering $\mathcal{Q}$ and 
of $p$-BAPU. 

The following embedding results for $\alpha$-modulation spaces are known. 
Suppose that $\alpha_1<\alpha_2<1$ and $1\leq p\leq \infty$. Then 
\begin{equation*} 
M^{s,\alpha_1}_{p,1}\subset M^{s,\alpha_2}_{p,1} 
\end{equation*} 
and, in particular, for any $\alpha<1$ and $1\leq p\leq \infty$ we have 
\begin{equation*} 
M^{s,\alpha}_{p,1}\subset B^{s}_{p,1} 
\end{equation*} 
The proofs of these results can be found e.g. in \cite{HW}; see Thm. 4.1 and Thm. 4.2. 

We next proceed to choose the initial vorticity $\omega_0$ in \eqref{B2:eq:euler-v}. 
Given any radial bump function $0 \leq \varphi \leq 1$ with support in $B(0,1/4)$ define 
\begin{align} \label{B2:eq:bump} 
\varphi_0(x_1, x_2) 
= 
\sum_{\varepsilon_1, \varepsilon_2 = \pm 1} 
\varepsilon_1 \varepsilon_2 \varphi(x_1 {-} \varepsilon_1, x_2 {-} \varepsilon_2). 
\end{align} 
Fix a positive integer $N_0 \in \mathbb{Z}_+$ and for any $M \gg 1$ let 
\begin{align} \label{B2:eq:iv} 
\omega_0(x) 
= 
M^{-2} N^{-\frac{1}{r}} \sum_{N_0 \leq k \leq N_0+N} \varphi_k(x),  
\end{align} 
where $1< p < \infty$, $N = 1, 2, 3 \dots$ and where 
\begin{align*} 
\varphi_k(x) = 2^{( - {1} + \frac{2}{r})k} \varphi_0 (2^k x). 
\end{align*} 
Clearly, the function $\varphi_0$ is odd in both variables $x_1, x_2$ 
and for any $k \geq 1$ the supports of $\varphi_k$ are disjoint and contained in a bounded set 
\begin{equation} \label{B2:eq:suppp} 
\mathrm{supp}\,{ \varphi_k } 
\subset 
\bigcup_{\varepsilon_1, \varepsilon_2 = \pm 1} 
B\big( (\varepsilon_1 2^{-k}, \varepsilon_2 2^{-k}), 2^{-(k+2)} \big). 
\end{equation} 
It is easy to check that $\omega_0 \in W^{1,r}(\mathbb{R}^2) \cap C^\infty_c(\mathbb{R}^2)$. 
In fact, we have 
\begin{lem} \label{B2:lem:omega-0} 
For any $2<r<\infty$ and any positive integer $N$, we have 
\begin{align} \label{B2:eq:omega-0} 
\| \omega_0 \|_{W^{1,r}} \lesssim  M^{-2} 
\end{align} 
\end{lem} 
\begin{proof} 
A straightforward calculation is omitted. 
\end{proof} 
%
%
%
%
%
%
%
%

Let $u = \nabla^\perp\Delta^{-1}\omega \in W^{2,r} \cap C^\infty$ be the associated velocity field and 
consider its Lagrangian flow $\eta(t)$, i.e., the solution of the initial value problem 
\begin{align} \label{B2:eq:flow} 
&\frac{d\eta}{dt}(t,x) = u(t, \eta(t,x)) 
\\  \nonumber 
&\eta(0,x) = x. 
\end{align} 
It can be checked that $\eta(t)$ is smooth and preserves the coordinate axes $x_1$, $x_2$ as well as 
the symmetries of the initial vorticity $\omega_0$ in \eqref{B2:eq:iv}. In fact, the flow is hyperbolic 
near the origin (a stagnation point) and we have the following 
\begin{prop} \label{B2:prop:Lag} 
Given $M \gg 1$ we have 
$$ 
\sup_{0 \leq t \leq M^{-3}} \| D\eta(t) \|_\infty > M 
$$ 
for any sufficiently large integer $N>0$ in \eqref{B2:eq:iv} and any $2<r<\infty$ 
sufficiently close to $2$. 
\end{prop} 
\begin{proof} 
See \cite{MY2}; Prop. 6. 
\end{proof} 
In order to proceed we need the following simple comparison result for the derivatives of solutions of 
the Lagrangian flow equations. 
\begin{lem} \label{B2:lem:comp} 
Let $u$ and $v$ be smooth divergence-free vector fields on $\mathbb{R}^2$. 
If $\eta$ and $\tilde\eta$ are the solutions of \eqref{B2:eq:flow} corresponding to $u$ and $u+v$ respectively, 
then 
$$ 
\sup_{0 \leq t \leq 1}{ \sum_{i=0,1} \big\| D^i\eta(t) - D^i\tilde\eta(t) \big\|_\infty } 
\leq 
C\sup_{0 \leq t \leq 1} \sum_{i=0,1} \big\| D^i v(t) \big\|_\infty 
$$ 
for some $C>0$ depending only on  the $L^\infty$ norms of $u$ and its derivatives. 
\end{lem} 
\begin{proof} 
Follows at once by applying Gronwall's inequality to the equation satisfied by 
the difference $\eta(t) - \tilde\eta(t)$. 
\end{proof} 
%

\section{The Proof of the main Theorem} 
\label{sec:proof} 

As in the previous section let $\omega(t) \in W^{1,r} \cap C^\infty$ be the solution of 
the vorticity equations \eqref{B2:eq:euler-v} 
with the initial condition \eqref{B2:eq:iv} and let $\eta(t)$ be the Lagrangian flow of 
the velocity field $u = \nabla^\perp \Delta^{-1}\omega$ as above. 
Our main goal in this section will be to prove 
\begin{theorem} \label{B2:thm:1} 
Let $r> 2$. Assume that the incompressible Euler equations are well-posed 
in the $\alpha$-modulation space 
$M^{1+\sigma,\alpha}_{p,q}(\mathbb{R}^2)$
for any $p\geq 2$, $q\geq 1$, $0 < \alpha \leq 1$ and $0 < \sigma < 1$ 
in the sense of Hadamard. 
Moreover, assume that $M^{1+\sigma,\alpha}_{p,q}$ is topologically embedded in $C^1(\mathbb{R}^2)$. 
Then there exist a $T>0$ and a sequence $\omega_{0,n}$ in $C^\infty_c$ 
with the following properties. 
\begin{itemize} 
\item[1.] 
There is a constant $C>0$ such that 
$\| \omega_{0,n} \|_{W^{1,r}} \leq C$ for sufficiently large positive integers $n$. 
\item[2.] 
For any $M \gg 1$ there is $0 < t_0 \leq T$ such that 
$\| \omega_{n}(t_0)\|_{W^{1,r}} \geq M^{1/3}$ 
for sufficiently large $n$ and for all $r$ sufficiently close to $2$.  
\end{itemize} 
\end{theorem} 
Since Hadamard's notion entails continuity of the data-to-solution map 
we deduce from Theorem \ref{B2:thm:1} that continuity cannot hold in $M^{1+s,\alpha}_{p,q}(\mathbb{R}^2)$ 
or else we get a contradiction with the following result. 
\begin{theorem**}[Kato-Ponce \cite{KP}] 
Let $1{<}r {<} \infty$ and $s {>}1{+} 2/r$. For any $\omega_0 \in W^{s-1, r}(\mathbb{R}^2)$ 
and any $T>0$ there exists a constant $K=K(T,\|\omega_0\|_{W^{s-1,r}})>0$ such that 
$$
\sup_{0 \leq t \leq T}\| \omega(t)\|_{W^{s-1,r}} \leq K. 
$$ 
\end{theorem**} 
Our Main Theorem will be a direct consequence of Theorem \ref{B2:thm:1}. 

\begin{proof}[Proof of Theorem \ref{B2:thm:1}] 
Given any large number $M \gg 1$ pick $T \leq M^{-3}$. Observe that if 
$\| \omega_0(t_0)\|_{W^{1,r}} > M^{1/3}$ 
for some $0 < t_0 \leq M^{-3}$ then there is nothing to prove and therefore we may assume that 
\begin{equation} \label{B2:eq:assump} 
\|\omega(t) \|_{W^{1,r}} \leq M^{1/3}, 
\qquad 
0 \leq t \leq M^{-3}. 
\end{equation} 
Next, by Proposition \ref{B2:prop:Lag} we can find $0 \leq t_0 \leq M^{-3}$ 
and a point $x^\ast = (x^\ast_1, x^\ast_2)$ in $\mathbb{R}^2$ 
for which the absolute value of one of the entries in the Jacobi matrix $D\eta(t_0,x^\ast)$ 
is at least as large as $M$. 
Because the velocity field $u$ is in $W^{2,r}$ so is the associated Lagrangian flow\footnote{E.g., by the wellposedness theory of \cite{KP}.} 
and hence by continuity in some sufficiently small $\delta$-neighbourhood of $x^\ast$ 
we have e.g., 
\begin{align} \label{B2:eq:M} 
\left| \frac{\partial \eta_2}{\partial x_2} (t_0,x) \right| > M 
\qquad 
\text{for all} 
\quad 
|x-x^\ast| < \delta. 
\end{align} 

We proceed to construct a sequence of high-frequency perturbations of the initial vorticity 
in $W^{1,r}$. 
To this end we choose a smooth bump function $0 \leq \hat{\chi} \leq 1$ with compact support 
in the unit ball $B(0,1)$ in the Fourier space and normalized by 
$\int_{\mathbb{R}^2} \hat\chi(\xi) \, d\xi = 1$. 
Using this function we set 
\begin{equation} \label{B2:eq:bp} 
\hat{\rho}(\xi) = \hat\chi(\xi - \xi_0) + \hat\chi(\xi + \xi_0), 
\qquad 
\xi \in \mathbb{R}^2, 
\quad 
\xi_0 =(2,0) 
\end{equation} 
so that $\mathrm{supp} \, \hat\rho \subset B(-\xi_0,1) \cup B(\xi_0,1)$ with 
\begin{equation} \label{B2:eq:ro2} 
\rho(0) = \int_{\mathbb{R}^2} \hat{\rho}(\xi) \, d\xi = 2 
\end{equation} 
and observe that for any $a >4$ we have 
\begin{equation} \label{B2:eq:sup-ro} 
\mathrm{supp}\, \hat\rho( \cdot \pm a, \cdot ) 
\cap 
B(0,1) = \emptyset. 
\end{equation} 
Define 
\begin{equation} \label{B2:eq:beta-pert} 
\beta_{k,\lambda}^{\tilde{\alpha},r}(x) 
= 
\frac{\lambda^{-1 + \frac{2}{r}}}{k^{1-\tilde\alpha}} 
\sum_{\varepsilon_1, \varepsilon_2 = \pm 1} 
\varepsilon_1 \varepsilon_2  \rho\big( \lambda(x-x^\ast_\epsilon) \big) \sin{kx_1}, 
\qquad 
k \in \mathbb{Z}^+, 
\; 
\lambda >0 
\end{equation} 
where 
$x^\ast_\epsilon = (\varepsilon_1 x^\ast_1, \varepsilon_2 x^\ast_2)$ 
and $\lambda > 0$ and $0 < \tilde \alpha < 1$ will be further specified below. 

\begin{remark}  
Note that the parameter $\lambda$ in \eqref{B2:eq:beta-pert} 
relates to the speed with which the support of the function $\rho$ is spreading in the Fourier space 
while $k$ expresses the speed of its translation. 
In the standard modulation space $M^{1+\sigma,0}_{\infty,1}(\mathbb{R}^2)$ 
one would need to set $\lambda = 1$ but in that case the spreading speed of the support of 
$\rho$ (its shrinking speed in physical space) 
is zero and hence the arguments we apply in the present paper break down. 
Therefore, the case of the standard modulation space $M^{1+\sigma,0}_{\infty, 1}(\mathbb{R}^2)$ 
remains an open problem. 
\end{remark} 

Before defining a suitable perturbation of $\omega_0$ we need to derive several estimates 
for $\beta_{k,\lambda}^{\tilde{\alpha},r}$ which we collect in the following lemma. 
\begin{lem} \label{B2:lem:rem}
Let $2 \leq p\leq \infty$, $2<r<\infty$ and $0<\sigma <1$. 
For any $k \in \mathbb{Z}^+$ and $\lambda >0$ sufficiently large, 
we have 
\begin{enumerate} 
\item[1.] 
$
\| \beta_{k,\lambda}^{\tilde{\alpha},r} \|_{W^{1,r}} 
\lesssim 
k^{-1+\tilde\alpha}+k^{\tilde\alpha}\lambda^{-1} 
$ 
\vskip 0.05cm 
\item[2.] 
$
\|\Delta^{1/2 + \sigma}\partial_j\Delta^{-1}\beta_{k,\lambda}^{\tilde{\alpha},r}\|_{L^p}
\lesssim 
k^{-1+\tilde\alpha}\lambda^{-1+2(1/r-1/p)}(\lambda^\sigma + k^\sigma) 
$
\vskip 0.05cm 
\item[3.] 
$ 
\|\partial_j \Delta^{-1} \beta_{k,\lambda}^{\tilde{\alpha},r} \|_{L^p} 
\lesssim 
k^{-1+\tilde\alpha}\lambda^{-1 + 2(1/r-1/p)} 
$ 
\end{enumerate} 
where $j = 1, 2$. 
\end{lem} 
\begin{proof} 
We need to compute the $L^r$ norms of $\beta_{k,\lambda}^{\tilde{\alpha},r}$ and its first derivative 
$\partial_i \beta_{k,\lambda}^{\tilde{\alpha},r}$. 
By the triangle inequality and the fact that $\hat\rho$ has compact support we have 
\begin{align*} 
\| \beta_{k,\lambda}^{\tilde{\alpha},r}\|_{L^r} 
\lesssim 
k^{-1+\tilde\alpha}\lambda^{-1}  
\sum_{\varepsilon_1, \varepsilon_2} 
\Bigg( \int_{\mathbb{R}^2} 
\Big| \rho\big( \lambda(x-x^\ast_\varepsilon) \big) \Big|^r \lambda^2 dx \Bigg)^{1/r} 
\lesssim 
k^{-1+\tilde\alpha} \lambda^{-1}. 
\end{align*} 
For the first derivatives, we have 
\begin{align*} 
\Big\| \frac{\partial\beta_{k,\lambda}^{\tilde{\alpha}.r}}{\partial x_1} \Big\|_{L^r} 
&\lesssim 
k^{-1+\tilde\alpha} \lambda^{2/r} 
\sum_{\varepsilon_1,\varepsilon_2} 
\Big\| 
\frac{\partial \rho}{\partial x_1}\big( \lambda(\cdot - x^*_\varepsilon) \big) 
\Big\|_{L^r} 
+ 
k^{\tilde\alpha}\lambda^{-1 + 2/r} 
 \sum_{\varepsilon_1,\varepsilon_2} 
\big\| 
\rho\big(\lambda(\cdot-x^*_\varepsilon)\big) 
\big\|_{L^r} 
\\ 
&\simeq 
k^{-1+\tilde\alpha} 
\Big\| \frac{\partial \rho}{\partial x_1}\Big\|_{L^r} 
+ 
k^{\tilde\alpha}\lambda^{-1} 
\| \rho\|_{L^r}  
\lesssim 
k^{-1+\tilde\alpha} + k^{\tilde\alpha}\lambda^{-1} 
\end{align*} 
and similarly 
\begin{align*} 
\Big\| \frac{\partial\beta_{k,\lambda}^{\tilde{\alpha},r}}{\partial x_2} \Big\|_{L^r} 
\lesssim 
k^{-1+\tilde\alpha} 
\Big\| \frac{\partial \rho}{\partial x_2} \Big\|_{L^r} 
\lesssim 
k^{-1+\tilde\alpha}. 
\end{align*} 
Combining these estimates gives the bound for $\| \beta_{k,\lambda}^{\tilde{\alpha},r} \|_{W^{1,r}}$. 

In order to derive the estimates in the remaining cases it will be convenient to use the Fourier transform 
\begin{align} \label{B2:eq:FTb} 
\hat{\beta}_{k,\lambda}^{\tilde{\alpha},r}(\xi) 
= 
\frac{\lambda^{-1+2/r}}{k^{1-\tilde\alpha}} 
\sum_{\varepsilon_1, \varepsilon_2 = \pm 1} \sum_{j=1,2} 
\varepsilon_1 \varepsilon_2 \frac{ (-1)^{j+1}}{2i\lambda^2} 
\hat{\rho} \big( \lambda^{-1} \xi_j^k \big) 
e^{-2\pi i \langle x_\varepsilon^\ast, \xi_j^k \rangle} 
\end{align} 
where $\xi_j^k = \big( \xi_1 + \frac{(-1)^j}{2\pi}k, \xi_2 \big)$. 
Let $p'$ be the conjugate exponent to $p$. Applying the Hausdorff-Young inequality we obtain 
\begin{align*} 
\big\| \Delta^{\frac{1 + \sigma}{2}}\partial_j \Delta^{-1} \beta_{k,\lambda}^{\tilde{\alpha},r} \big\|_{L^p} 
&\lesssim
\big\| |\cdot|^\sigma \hat\beta_{k,\lambda}^{\tilde{\alpha},r} \big\|_{L^{p'}} 
\\ 
&\lesssim 
k^{-1+\tilde\alpha}\lambda^{-1+2/r}\sum_{j=1,2}\left(\int_{\mathbb{R}^2}\lambda^{-2p'} 
|\xi|^{\sigma p'} \big| \hat{\rho}(\lambda^{-1}\xi_j^k) \big|^{p'} \, d\xi\right)^{1/p'} 
\end{align*} 
and changing the variables we further estimate by 
\begin{align*} 
&\lesssim 
k^{-1+\tilde\alpha}\lambda^{-1+\frac{2}{r}-2(1-\frac{1}{p'})} 
\sum_{j=1,2} 
\left( \int_{\mathbb{R}^2} 
\bigg( 
\Big( \eta_1-\frac{(-1)^j}{2\pi}k \Big)^2
+ 
\eta_2^2 \bigg)^{\frac{\sigma p'}{2}} 
\big| \hat{\rho}(\lambda^{-1}\eta) \big|^{p'} 
\frac{d\eta}{\lambda^2} \right)^{1/p'} 
\\ 
&\lesssim 
k^{-1+\tilde\alpha}\lambda^{-1+\frac{2}{r}-\frac{2}{p}} 
\sum_{j=1,2} 
\left( \int_{\mathbb{R}^2} 
\bigg( 
\Big(\lambda\zeta_1-\frac{(-1)^j}{2\pi}k \Big)^2 
+ 
(\lambda\zeta_2)^2 \bigg)^{\frac{\sigma p'}{2}} \big| \hat{\rho}(\zeta) \big|^{p'} 
d\zeta \right)^{1/p'} 
\\ 
&\lesssim 
k^{-1+\tilde\alpha}\lambda^{-1+2\left(\frac{1}{r}-\frac{1}{p}\right)} 
\big( \lambda^\sigma + k^\sigma \big) 
\end{align*} 
for any $\sigma \geq 0$. 

By the same calculation as above, we also have 
\begin{align*} 
\big\| \partial_j &\Delta^{-1} \beta_{k,\lambda}^{\tilde{\alpha},r} \big\|_{L^p} 
\lesssim
\big\| |\cdot|^{-1}\hat\beta_{k,\lambda} \big\|_{L^{p'}} 
\\
&\lesssim 
k^{-1+\tilde\alpha} \lambda^{-1+\frac{2}{r}-\frac{2}{p}} 
\sum_{j=1,2} \left(\int_{\mathbb{R}^2} 
\left( \Big(\lambda\zeta_1-\frac{(-1)^j}{2\pi}k \Big)^2 
+ 
(\lambda\zeta_2)^2\right)^{-\frac{p'}{2}} \big| \hat{\rho}(\zeta) \big|^{p'} 
d\zeta \right)^{1/p'} 
\\ 
&\lesssim 
k^{-1+\tilde\alpha}\lambda^{-2+2\left(\frac{1}{r}-\frac{1}{p}\right)}
\end{align*} 
for sufficiently large $k$ and $\lambda$. 
\end{proof} 
From now on we will restrict to the case 
\begin{equation} \label{B2:eq:kln} 
\lambda = k^{\tilde\alpha}, 
\quad 
k = n 
\quad \text{and} \quad 
0 < \tilde{\alpha} \leq \alpha \leq 1 
\end{equation} 
and observe that it is possible to choose the integers $n$ are sufficiently large so that, in particular, 
the assumptions of the previous lemma hold. 
Let $\beta_n = \beta_{k,\lambda}^{\tilde{\alpha},r}$ and define a sequence of 
initial vorticities by 
\begin{equation} \label{B2:eq:omega-seq} 
\omega_{0,n}(x) = \omega_0(x) + \beta_n(x), 
\qquad 
n \gg 10. 
\end{equation} 
Combining the first part of Lemma \ref{B2:lem:rem} with equation \eqref{B2:eq:omega-0} of 
Lemma \ref{B2:lem:omega-0} we find that $\omega_{0,n}$ belong to $W^{1,r}$, namely 
\begin{align} \label{B2:eq:omega-0n} 
\| \omega_{0,n} \|_{W^{1,r}} 
\leq 
\| \omega_0\|_{W^{1,r}} + \| \beta_n \|_{W^{1,r}} 
\lesssim 
1 
\end{align} 
for any sufficiently large $n$. This proves the first assertion of Theorem \ref{B2:thm:1}. 

\smallskip 
Denote by $\omega_n(t)$ the sequence of vorticity solutions of \eqref{B2:eq:euler-v} 
with initial data $\omega_{0,n}$ and as before let $\eta_n(t)$ be the Lagrangian flows of 
the corresponding velocity fields $u_n = \nabla^\perp\Delta^{-1}\omega_n$. 
The following lemma will be crucial in what follows. 
\begin{lem} \label{lem:fi} 
Let $0 < \sigma <1$, $0 < \alpha \leq 1$ and $2 \leq p \leq \infty$. 
For any $1 \leq q \leq \infty$ we have 
\begin{align*} 
\| \nabla^\perp \Delta^{-1} \beta_n \|_{M^{1+ \sigma ,\alpha}_{p,q}}  
\simeq  
\|\nabla^\perp \Delta^{-1} \beta_n \|_{L^p} 
+
\| \Delta^{\frac{1 + \sigma}{2}} \nabla^\perp \Delta^{-1} \beta_n \|_{L^p} 
\end{align*} 
for any sufficiently large $n \in \mathbb{Z}^{+}$. 
\end{lem} 
\begin{proof} 
From \eqref{B2:eq:sup-ro} and \eqref{B2:eq:FTb} we see that for any sufficiently large integer 
$n \in \mathbb{Z}^{+}$ the subsets $\mathrm{supp}\, \hat{\beta}_n$ and $B(0,1)$ are disjoint. 
Thus, it suffices to consider the case $\sigma=0$, that is 
$$
\| \beta_n \|_{L^p} \simeq \| \beta_n \|_{M^{0,\alpha}_{p,q}}. 
$$ 
Let $\mathcal Q$ be an admissible $\alpha$-covering by sets of size $|Q|\sim(1+|x|^2)^{\alpha}$. 
Using \eqref{B2:eq:FTb}, \eqref{B2:eq:kln} and the fact that $\mathrm{supp}\,\hat{\rho} \subset B(0,3)$ 
we find 
$$ 
\mathrm{supp}\, \hat{\beta}_{n=2^j} 
\subset 
B\big( (2^{j},0), 2^{\tilde\alpha j} \big) \subset B\big( (2^{j},0), 2^{\alpha j} \big) 
$$ 
so that for any $j \in \mathbb{Z}_{+}$ there is a $Q\in\mathcal Q$ with 
$\mathrm{supp}\, \hat{\beta}_{2^j} \subset Q$ 
and we have 
$$
\| \beta_{2^j} \|_{M^{0,\alpha}_{p,q}} 
= 
\left(\sum_{Q\in\mathcal Q} \|\mathcal F^{-1}\psi_Q\mathcal F\beta_{2^j} \|_p^q\right)^{1/q}
\simeq 
\| \beta_{2^j} \|_{L^p} 
$$
for $q<\infty$. Note that the case $q=\infty$ is analogous.
%
%
\end{proof} 

Suppose now that the data-to-solution map for the Euler equations \eqref{E} is continuous 
from bounded subsets in $M^{1+\sigma,\alpha}_{p,q}(\mathbb{R}^2)$ into 
$C\big( [0,1], M^{1+\sigma,\alpha}_{p,q}(\mathbb{R}^2) \big)$. 
Choose $0 < \tilde \alpha \leq \alpha$ so that 
\begin{equation*}
-1+ \sigma +2\tilde\alpha(1/r-1/p)<0.
\end{equation*}
Then, from estimates 2 and 3 of Lemma \ref{B2:lem:rem} we have  
\begin{align*} 
\| \nabla^\perp \Delta^{-1} \beta_n \|_{L^p}  
+ 
\|\Delta^{\frac{1+ \sigma}{2}} \nabla^\perp \Delta^{-1} \beta_n \|_{L^p} 
\rightarrow 0 
\quad 
\text{as} 
\; 
n \to \infty 
\end{align*} 
where $\beta_n$ is the perturbation sequence defined in \eqref{B2:eq:beta-pert} 
and combining \eqref{B2:eq:omega-seq} with Lemma \ref{lem:fi} 
we obtain 
\begin{align} \label{B2:eq:n1} 
\| \nabla^\perp \Delta^{-1} ( \omega_{0,n} - \omega_0) \|_{M^{1+\sigma,\alpha}_{p,q}} 
\longrightarrow 0 
\quad 
\text{as} 
\; 
n \to \infty. 
\end{align} 
The continuity assumption on the solution map and \eqref{B2:eq:n1} now imply 
\begin{equation} \label{B2:eq:A} 
\sup_{0 \leq t \leq T}\| \nabla^\perp\Delta^{-1} ( \omega_n(t) - \omega(t) ) \|_{M^{1+\sigma,\alpha}_{p,q}} 
\longrightarrow 0 
\quad 
\text{as} 
\; 
n \to \infty 
\end{equation} 
from which, using the embedding assumption $M^{1+\sigma}_{p,q} \subset C^1$ and 
Lemma \ref{B2:lem:comp}, we obtain 
\begin{align} \label{B2:eq:LL} 
\sup_{0\leq t \leq T} \sum_{i=0,1} 
\| D^i\eta_n(t) - D^i\eta(t) \|_\infty 
\longrightarrow 0 
\quad 
\text{as} 
\; 
n \to \infty 
\end{align} 
where as before $\eta(t)$ is the Lagrangian flow of the (smooth) divergence-free vector field 
$u=\nabla^\perp\Delta^{-1}\omega$ of Proposition \ref{B2:prop:Lag} and 
$\eta_n(t)$ is the flow of $u_n = \nabla^\perp\Delta^{-1}\omega_n$ 
whose initial vorticities are given by  \eqref{B2:eq:iv} and \eqref{B2:eq:omega-seq} respectively. 
The rest of the argument is completely analogous to the proof of Theorem 3 in \cite{MY2}. 
Thus we omit the details. 
\end{proof}

\section{A direct proof for $C^{1+\sigma}(\mathbb{R}^3)$ based on shear flow} 
\label{shear flow} 

In this section we present a short and direct argument showing the loss of continuity of 
the data-to-solution map of \eqref{E} in the classical H\"older space $C^{1+\sigma}$ with $0 < \sigma <1$. 
It is inspired by conversations with A. Shnirelman and C. Bardos from whom we learned 
about the DiPerna-Majda shear flow techniques. 

Consider two 3D shear flows 
$$ 
u(t,x) = \big( f(x_2), 0, h(x_1 - tf(x_2)) \big) 
\quad \text{and} \quad 
v(t,x) = \big( g(x_2), 0, h(x_1 - tg(x_2)) \big). 
$$ 
Let $f$, $g$ and $h$ be bounded functions in $C^{1+ \sigma}(\mathbb{R}^3)$ 
with $h$ chosen so that in addition its derivative satisfies 
$$ 
h'(x_1) = |x_1|^\sigma 
\qquad \text{for} \quad 
-2a \leq x_1 \leq 2a 
$$  
where $a=\max\{\sup_{x_1}|f(x_1)|,\sup_{x_1}|g(x_1)|\}$.
It is easy to verify that both $u(t)$ and $v(t)$ solve the 3D Euler equations 
with initial conditions 
$$
u_0(x) = \big( f(x_2), 0, h(x_1) \big) 
\quad \text{and} \quad 
v_0(x) = \big( g(x_2), 0, h(x_1) \big). 
$$ 

Now, given any $\epsilon>0$ adjust $f$ and $g$ so that 
\begin{equation*}
\| u_0 - v_0 \|_{C^{1+\sigma}} = \| f - g \|_{C^{1+\sigma}} < \epsilon 
\end{equation*}
and consider at any time $0 \leq t \leq 1$ the norm of the difference of the corresponding solutions 
\begin{align*}
\| u(t) - v(t) \|_{C^{1+\sigma}} 
&= 
\| f - g \|_{C^{1+\sigma}} 
+ 
\big\| h( \cdot - tf(\cdot)) - h( \cdot - tg( \cdot )) \big\|_{C^{1+\sigma}} 
\\ 
&\geq 
\big\| \nabla \big( h( \cdot - tf( \cdot )) - h(\cdot - tg(\cdot)) \big) \big\|_{C^{\sigma}} 
\\
&= 
\big\| h'( \cdot - tf(\cdot)) - h'(\cdot - tg(\cdot)) \big\|_{C^{\sigma}} 
\end{align*} 
which can be further bounded from below by 
\begin{align*} 
\geq 
\sup_{\stackrel{x \neq y}{x,y\in[-b,b]^2}} 
\frac{ \big| 
\big( |x_1 - tf(x_2)|^\sigma - |x_1 - tg(x_2)|^\sigma \big) 
- 
\big( |y_1 - tf(y_2)|^\sigma - |y_1 - tg(y_2)|^\sigma \big) 
\big|}{ | x - y |^\sigma } 
\end{align*}
where $b=\min\{\sup_{x_1}|f(x_1)|,\sup_{x_1}|g(x_1)|\}$. 

Finally, pick $x_2 = y_2 = c$ for some arbitrary constant $c$ so that 
the expression above becomes 
\begin{equation*}
\geq 
\sup_{\stackrel{x_1 \neq y_1}{x,y\in[-b,b]^2} }
\frac{ 
|( |x_1 - tf(c)|^\sigma - |x_1 - tg(c)|^\sigma ) 
- 
( |y_1 - tf(c)|^\sigma - |y_1 - tg(c)|^\sigma )| 
}{ |x_1 - y_1|^\sigma } 
\end{equation*}
and evaluate it once again from below by choosing 
$x_1 = tg(c)$ and $y_1 = tf(c)$ 
to get the bound 
\begin{equation*}
\geq 
\frac{ t^\sigma |g(c) - f(c)|^\sigma + t^\sigma |f(c) - g(c)|^\sigma }{ t^\sigma |f(c) - g(c)|^\sigma } 
= 2. 
\end{equation*}
This shows local ill-posedness of the 3D Euler equations in $C^{1+\sigma}$ in the Hadamard sense 
considered here. 
$\square$ 

\vspace{0.5cm}
\noindent
{\bf Acknowledgments.}\ 
Part of this work was done while GM was the Ulam Chair Visiting Professor 
at the University of Colorado, Boulder. 
TY was partially supported by JSPS KAKENHI Grant Number 25870004. 
Both authors would like to thank Professors Yasushi Taniuchi and Alexander Shnirelman 
for many inspiring conversations that led to this paper. 

\bibliographystyle{amsplain}

\begin{thebibliography}{10} 


\bibitem{BCD} 
H. Bahouri, J. Chemin and R. Danchin, 
\textit{Fourier Analysis and Nonlinear Partial Differential Equations}, 
Springer, New York 2011. 

\bibitem{BT} 
C. Bardos and I. Titi, 
\textit{Loss of smoothness and energy conserving rough weak solutions for the 3d Euler equations}, 
Discrete Cont. Dyn. Syst. ser. \textbf{S3} (2010), 185-197. 

\bibitem{BL} 
J. Bourgain and D. Li, 
\textit{Strong ill-posedness of the incompressible Euler equations in borderline Sobolev spaces}, 
to appear in Invent. math. 

\bibitem{BL1} 
J. Bourgain and D. Li, 
\textit{Strong illposedness of the incompressible Euler equation in integer $C^m$ spaces}, 
preprint arXiv:1405.2847 [math.AP]. 


\bibitem{Co} 
P. Constantin, 
\textit{On the Euler equations of incompressible fluids}, 
Bull. Amer. Math. Soc. \textbf{44} (2007), 603-621. 

\bibitem{CS} 
A. Cheskidov and R. Shvydkoy, 
\textit{Ill-posedness of basic equations of fluid dynamics in Besov spaces}, 
Proc. A.M.S. \textbf{138} (2010), 1059-1067. 



\bibitem{EM}
D. Ebin and J. Marsden, 
\textit{Groups of diffeomorphisms and the motion of an incompressible fluid}, 
Ann. Math. \textbf{92} (1970), 102-163.  

\bibitem{EM} 
T. Elgindi and N. Masmoudi, 
\textit{$L^\infty$ ill-posedness for a class of equations arising in hydrodynamics}, 
preprint arXiv:1405.2478 [math.AP]. 

\bibitem{FHW} 
H. G. Feichtinger, C. Huang and B. Wang,
\textit{Trace operators for modulation, $\alpha$-modulation and Besov spaces}, 
Appl. Comput. Harmon. Anal. \textbf{30} (2011), 110-127. 

\bibitem{Gu} 
N. Gyunter, 
\textit{On the motion of a fluid contained in a given moving vessel}, 
(Russian), Izvestia AN USSR, Sect. Phys. Math. (1926-8). 

\bibitem{HW} 
J. Han and B. Wang,
\textit{$\alpha$-modulation spaces (I) embedding, interpolation and algebra properties}, 
 J. Math. Soc. Japan \textbf{66} (2014), 1315-1373. 

\bibitem{KL} 
T. Kato and C. Lai, 
\textit{Nonlinear evolution equations and the Euler flow}, 
J. Funct. Anal. \textbf{56} (1984), 15-28. 

\bibitem{KP0} 
T. Kato and G. Ponce, 
\textit{Well-posedness of the Euler and Navier-Stokes equations in the Lebesgue spaces 
$L^p_s(\mathbf{R}^2)$}, 
Rev. Mat. Iberoamericana \textbf{2} (1986), 73-88. 

\bibitem{KP} 
T. Kato and G. Ponce, 
\textit{On nonstationary flows of viscous and ideal fluids in $L^p_s(\mathbb{R}^2)$}, 
Duke Math. J. \textbf{55} (1987), 487-499. 


\bibitem{Li} 
L. Lichtenstein, 
\textit{Uber einige Existenzprobleme der Hydrodynamik unzusamendruckbarer, 
reibunglosiger Flussigkeiten und die Helmholtzischen Wirbelsatze}, 
Math. Zeit. \textbf{23} (1925), \textbf{26} (1927), \textbf{28} (1928), \textbf{32} (1930). 

\bibitem{MB} 
A. Majda and A. Bertozzi, 
\textit{Vorticity and Incompressible Flow}, 
Cambridge University Press, Cambridge 2002. 



\bibitem{MY} 
G. Misio{\l}ek and T. Yoneda, 
\textit{Ill-posedness examples for the quasi-geostrophic and the Euler equations}, 
Analysis, geometry and quantum field theory, Contemp. Math. \textbf{584}, 
Amer. Math. Soc., Providence, RI, 2012, 251-258. 


\bibitem{MY2} 
G. Misio{\l}ek and T. Yoneda, 
\textit{Local ill-posedness of the incompressible Euler equations in $C^1$ and $B^1_{\infty,1}$}, 
preprint arXiv:1405.1943 [math.AP] and arXiv:1405.4933 [math.AP] (2014). 

\bibitem{MSY} 
A. Morgulis, A. Shnirelman and V. Yudovich, 
\textit{Loss of smoothness and inherent instability of 2D inviscid fluid flows}, 
Comm. Partial Differential Eqns \textbf{33} (2008), 943-968. 

\bibitem{PP} 
H. Pak and Y. Park, 
\textit{Existence of solution for the Euler equations in a critical Besov space $B^1_{\infty, 1}(\mathbb{R}^n)$},  Comm. Partial Differential Equations \textbf{29} (2004), 1149-1166.



\bibitem{Ta} 
R. Takada, 
\textit{Counterexamples of commutator estimates in the Besov and the Triebel-Lizorkin spaces 
related to the Euler equations}, 
SIAM J. Math. Anal. \textbf{42} (2010), 2473-2483. 


\bibitem{V}
M. Vishik,
\textit{Hydrodynamics in Besov spaces}, 
Arch. Rational Mech. Anal. \textbf{145} (1998), 197-214.

\bibitem{Wo} 
W. Wolibner, 
\textit{Un theor\'eme sur l'existence du mouvement plan d'un fluide parfait, homog\'ene, 
incompressible, pendant un temps infiniment long}, 
Math. Z. \textbf{37} (1933), 698-726. 

\bibitem{Yu} 
V. Yudovich, 
\textit{Non-stationary flow of an incompressible liquid}, 
Zh. Vychis. Mat. Mat. Fiz. \textbf{3} (1963), 1032-1066. 

\end{thebibliography}

\end{document}